\documentclass{scrartcl}
\usepackage{amssymb,amsmath,amsthm}
\usepackage{dsfont}
\usepackage{float}
\usepackage{color}
\usepackage{tikz}
\usetikzlibrary{automata,positioning,arrows}
\usepackage{authblk}
\newtheorem{Theorem}{Theorem}[section] 
\newtheorem{Lemma}[Theorem]{Lemma}     
\newtheorem{Corollary}[Theorem]{Corollary}
\newtheorem{Definition}[Theorem]{Definition}
\newtheorem{Proposition}[Theorem]{Proposition}
\newtheorem{Example}[]{Example}
\newtheorem{Remark}{Remark}

\usepackage[square, sort, comma, numbers]{natbib} 

\renewcommand{\tilde}{\widetilde}

\newcommand{\hQ}{\hat{\MQ}}
\newcommand{\hk}{\hat{k}}

\newcommand{\BT}{\mathbb{T}}
\newcommand{\BU}{\mathbb{U}}

\newcommand{\MA}{\mathds{A}}

\newcommand{\MN}{\mathds{N}}
\newcommand{\MZ}{\mathds{Z}}
\newcommand{\MQ}{\mathds{Q}}

\newcommand{\MR}{\mathds{R}}

\newcommand{\MH}{\mathds{H}}

\newcommand{\CO}{\mathcal{O}}

\newcommand{\ind}{\mathds{1}}
\newcommand{\GG}{\mathbb{G}}
\newcommand{\TW}{\widetilde{W}}

\newcommand{\fp}{\mathfrak{p}}

\newcommand{\fq}{\mathfrak{q}}

\newcommand{\UT}{\underline{T}}

\newcommand{\Algcat}[1]{#1\text{-}\mathsf{Alg}}

\newcommand{\ds}{/\!\!/}
\newcommand{\genus}{\mathrm{genus}}
\newcommand{\mass}{\mathrm{mass}}

\newcommand{\GL}{\mathrm{GL}}

\newcommand{\Sp}{\mathrm{Sp}}

 \newcommand{\U}{\mathrm{U}} 
\newcommand{\stab}{\operatorname{Stab}} 
 \newcommand{\Hom}{\operatorname{Hom}}

\newcommand{\End}{\mathrm{End}}

\newcommand{\An}[1]{ \scalebox{#1}{\begin{tikzpicture}
    \draw (0,0) -- (2.5,0);
    \draw (3.5,0) -- (6,0);
    \draw (0,0) -- (3,2);
    \draw (6,0) -- (3,2);

    \draw[fill=white] (0,0) circle(.1);
    \draw[fill=white] (2,0) circle(.1);
    \draw[fill=white] (4,0) circle(.1);
    \draw[fill=white] (6,0) circle(.1);
    \draw[fill=white] (3,2) circle(.1);

    \node at (0,0.35) {$1$};
    \node at (2,0.35) {$2$};
    \node at (4,0.35) {$n-1$};
    \node at (6,0.35) {$n$};
    \node at (3,2.35) {$0$};
    \node at (3.0,0) {$\hdots$};
\end{tikzpicture}}}

\newcommand{\Bn}[1]{\scalebox{#1}{
 \begin{tikzpicture}   
    \draw (0.5,1) -- (2,0);
    \draw (0.5,-1) -- (2,0);
    \draw (2,0) -- (4.5,0);
    \draw (5.5,0) -- (6,0);
    \draw (6,0.07) -- (8,0.07);
    \draw (6,-0.07) -- (8,-0.07);
    \draw (6.9,0.2) -- (7.1,0);
    \draw (6.9,-0.2) -- (7.1,0);
    \draw[fill=white] (0.5,1) circle(.1);
    \draw[fill=white] (0.5,-1) circle(.1);
    \draw[fill=white] (2,0) circle(.1);
    \draw[fill=white] (4,0) circle(.1);
    \draw[fill=white] (6,0) circle(.1);
    \draw[fill=white] (8,0) circle(.1);
    
    \node at (0.5,1.35) {$0$};
    \node at (0.5,-0.65) {$1$};
    \node at (2,0.35) {$2$};
    \node at (4,0.35) {$3$};
    \node at (6,0.35) {$n-1$};
    \node at (8,0.35) {$n$};
    \node at (5.0,0) {$\hdots$};
\end{tikzpicture}}
}

\newcommand{\Cn}[1]{\scalebox{#1}{
\begin{tikzpicture}
    \draw(0,0.07) -- (2,0.07);
    \draw(0,-0.07) -- (2,-0.07);
    \draw (1.1,0) -- (0.9,0.2);
    \draw (1.1,0) -- (0.9,-0.2);
    \draw (2,0) -- (4.5,0);
    \draw (5.5,0) -- (6,0);
    \draw (6,0.07) -- (8,0.07);
    \draw (6,-0.07) -- (8,-0.07);
    \draw (6.9,0) -- (7.1,0.2);
    \draw (6.9,0) -- (7.1,-0.2);
    \draw[fill=white] (0,0) circle(.1);
    \draw[fill=white] (2,0) circle(.1);
    \draw[fill=white] (4,0) circle(.1);
    \draw[fill=white] (6,0) circle(.1);
    \draw[fill=white] (8,0) circle(.1);
    
    \node at (0,0.35) {$0$};
    \node at (2,0.35) {$1$};
    \node at (4,0.35) {$2$};
    \node at (6,0.35) {$n-1$};
    \node at (8,0.35) {$n$};
    \node at (5.0,0) {$\hdots$};
\end{tikzpicture}}
}

\newcommand{\Dn}[1]{\scalebox{#1}{
 \begin{tikzpicture}
    \draw (0.5,1) -- (2,0);
    \draw (0.5,-1) -- (2,0);
    \draw (2,0) -- (4.5,0);
    \draw (5.5,0) -- (6,0);
    \draw (6,0) -- (7.5,1);
    \draw (6,0) -- (7.5,-1);

    \draw[fill=white] (0.5,1) circle(.1);
    \draw[fill=white] (0.5,-1) circle(.1);
    \draw[fill=white] (2,0) circle(.1);
    \draw[fill=white] (4,0) circle(.1);
    \draw[fill=white] (6,0) circle(.1);
    \draw[fill=white] (7.5,1) circle(.1);
    \draw[fill=white] (7.5,-1) circle(.1);

    \node at (0.5,1.35) {$0$};
    \node at (0.5,-0.65) {$1$};
    \node at (2,0.35) {$2$};
    \node at (4,0.35) {$3$};
    \node at (5.9,0.35) {$n-2$};
    \node at (7.5,1.35) {$n-1$};
    \node at (7.5,-0.65) {$n$};
    \node at (5.0,0) {$\hdots$};
\end{tikzpicture}}
}

\newcommand{\Esix}[1]{\scalebox{#1}{ \begin{tikzpicture}
    \draw (0,0) -- (8,0);
    \draw (4,0) -- (4,2);

    \draw[fill=white] (0,0) circle(.1);
    \draw[fill=white] (2,0) circle(.1);
    \draw[fill=white] (4,0) circle(.1);
    \draw[fill=white] (6,0) circle(.1);
    \draw[fill=white] (8,0) circle(.1);
    \draw[fill=white] (4,1) circle(.1);
    \draw[fill=white] (4,2) circle(.1);

    \node at (4.35,2) {$0$};
    \node at (4.35,1) {$2$};
    \node at (0,0.35) {$1$};
    \node at (2,0.35) {$3$};
    \node at (4.35,0.35) {$4$};
    \node at (6,0.35) {$5$};
    \node at (8,0.35) {$6$};
\end{tikzpicture}}
}

\newcommand{\Eseven}[1]{ \scalebox{#1}{\begin{tikzpicture}
    \draw (0,0) -- (9,0);
    \draw (4.5,0) -- (4.5,1);

    \draw[fill=white] (0,0) circle(.1);
    \draw[fill=white] (1.5,0) circle(.1);
    \draw[fill=white] (3,0) circle(.1);
    \draw[fill=white] (4.5,0) circle(.1);
    \draw[fill=white] (6,0) circle(.1);
    \draw[fill=white] (7.5,0) circle(.1);
    \draw[fill=white] (9,0) circle(.1);
    \draw[fill=white] (4.5,1) circle(.1);

    \node at (0,0.35) {$0$};
    \node at (1.5,0.35) {$1$};
    \node at (3,0.35) {$3$};
    \node at (4.85,0.35) {$4$};
    \node at (6,0.35) {$5$};
    \node at (7.5,0.35) {$6$};
    \node at (9,0.35) {$7$};
    \node at (4.85,1) {$2$};
\end{tikzpicture}}
 }
\newcommand{\Eeight}[1]{\scalebox{#1}{
 \begin{tikzpicture}   
    \draw (0,0) -- (10.5,0);
    \draw (7.5,0) -- (7.5,1);

    \draw[fill=white] (0,0) circle(.1);
    \draw[fill=white] (1.5,0) circle(.1);
    \draw[fill=white] (3,0) circle(.1);
    \draw[fill=white] (4.5,0) circle(.1);
    \draw[fill=white] (6,0) circle(.1);
    \draw[fill=white] (7.5,0) circle(.1);
    \draw[fill=white] (9,0) circle(.1);
    \draw[fill=white] (10.5,0) circle(.1);
    \draw[fill=white] (7.5,1) circle(.1);

    \node at (0,0.35) {$0$};
    \node at (1.5,0.35) {$1$};
    \node at (3,0.35) {$3$};
    \node at (4.5,0.35) {$4$};
    \node at (6,0.35) {$5$};
    \node at (7.85,0.35) {$6$};
    \node at (9,0.35) {$7$};
    \node at (10.5,0.35) {$8$};
    \node at (7.85,1) {$2$};
\end{tikzpicture}}
}
\newcommand{\Ffour}[1]{ \scalebox{#1}{
\begin{tikzpicture}
    \draw (0,0) -- (4,0);
    \draw (4,0.07) -- (6,0.07);
    \draw (4,-0.07) -- (6,-0.07);
    \draw (6,0) -- (8,0);
    \draw (4.9,0.2) -- (5.1,0);
    \draw (4.9,-0.2) -- (5.1,0);

    \draw[fill=white] (0,0) circle(.1);
    \draw[fill=white] (2,0) circle(.1);
    \draw[fill=white] (4,0) circle(.1);
    \draw[fill=white] (6,0) circle(.1);
    \draw[fill=white] (8,0) circle(.1);

    \node at (0,0.35) {$0$};
    \node at (2,0.35) {$1$};
    \node at (4,0.35) {$2$};
    \node at (6,0.35) {$3$};
    \node at (8,0.35) {$4$};
\end{tikzpicture}}
}
\newcommand{\Gtwo}[1]{\scalebox{#1}{
\begin{tikzpicture}
    \draw (2,0.07) -- (4,0.07);
    \draw (0,0) -- (4,0);
    \draw (2,-0.07) -- (4,-0.07);
    \draw (3.1,0) -- (2.9,0.2);
    \draw (3.1,0) -- (2.9,-0.2);
    \draw[fill=white] (0,0) circle(.1);
    \draw[fill=white] (2,0) circle(.1);
    \draw[fill=white] (4,0) circle(.1);
    
    \node at (0,0.35) {$0$};
    \node at (2,0.35) {$1$};
    \node at (4,0.35) {$2$};
\end{tikzpicture}}
}

\title{Simultaneous computation of Hecke operators} 

 \author{Sebastian Sch\"onnenbeck\thanks{The author is supported by the DFG research training group \emph{Experimental and constructive algebra} (GRK 1632).} }
\affil{RWTH Aachen University\\
   Lehrstuhl D f\"ur Mathematik\\
   Pontdriesch 14/16, 52052 Aachen\\
   Germany\\
  \texttt{sebastian.schoennenbeck@rwth-aachen.de}
}

\begin{document}
\maketitle

\begin{abstract}
We present a method to compute two Hecke operators acting on a space of algebraic modular forms simultaneously based on an idea of Eichler's. We show that in certain cases this method can be used to obtain the action of the full Hecke algebra with respect to a hyperspecial subgroup and use it to compute Hecke eigenforms for compact forms of symplectic groups.
\end{abstract}


\section{Introduction} 
A classical task in the arithmetic theory of quadratic forms is the enumeration of a system of representatives of the isometry classes in a genus, i.e. given a quadratic space $(V,q)$ over $\MQ$ and a $\MZ$-lattice $L \subset V$ decompose the set of all latices that are locally isometric to $L$ at every prime $p$ into (global) isometry classes. The number of isometry classes is known to be finite (even completely known without computation if $q$ is indefinite) and the task is usually settled by use of Kneser's $p$-neighbour method (cf. \cite{KneserKlassenzahlenDefinit}). 

Let us now replace the orthogonal group (with respect to $q$) with another reductive linear algebraic group $\GG$ over $\MQ$. In this situation we can still ask the same question, i.e. given a faithful representation $\GG \hookrightarrow \GL_n$ and a lattice $L \subset \MQ^n$ we would like to decompose the $\GG(\hQ)$-orbit of $L$ (called the $\GG$-genus) into $\GG(\MQ)$-orbits (called $\GG$-classes), where $\hQ$ denotes the finite adeles of $\MQ$. Again this question is well-studied, for example it is known that the class number is one, if $\GG$ is simply connected, absolutely simple and $\GG(\MR)$ is not compact by virtue of the strong approximation property. On the other hand if $\GG(\MR)$ is compact there are analogues of Kneser's neighbour method that allow us to tackle this problem algorithmically (cf. \cite{GreenbergVoightLatticeMethods} and the series \cite{ASAPquadratic,ASAPclassical,ASAPexceptional}). A helpful tool in the latter case is provided by so-called mass formulae. Since $\GG(\MR)$ is assumed to be compact the stabilizer of a lattice in $\GG(\MQ)$ is finite and hence so is the quantity
\begin{equation}
 \mass(L)=\mass(\genus(L))=\sum_M \frac{1}{|\stab_{\GG(\MQ)}(M)|}
\end{equation}
where the sum runs over a system of representatives of the $\GG$-classes in the $\GG$-genus of $L$. Since this quantity only depends on local information on $L$ it is (as long as $\GG$ and $L$ are suitably well-behaved) computable without actually writing down a system of representatives. The probably best-known instance of this principle is the Smith-Minkowski-Siegel mass formula in the case of orthogonal groups (\cite{SiegelAnalytischeTheorie}) which was later generalized to the case of arbitrary classical groups (\cite{GanYuGroupSchemes}). These mass formulae often a priori only compute the mass of quite restrictive genera (e.g. even unimodular lattices in the case of orthogonal group), hence it is desirable to be able to compare the masses of different genera. The idea how to do this goes back to the work of Eichler (cf. \cite{EichlerAehnlichkeitsklassen}) and works as follows: Let $L,L'$ be two lattices in faithful $\GG$-modules and let $K,K'$ be their respective stabilizers in $\GG(\hQ)$, then
\begin{equation}
 \mass(L)[K:(K \cap K')]=\mass(L')[K':(K \cap K')].
\end{equation}
Moreover the idea behind this can be used to actually compute representatives for certain genera starting from representatives of another genus (cf. \cite{NebeBachoc}). 

In this article we want to apply Eichler's method to the computation of Hecke operators acting on algebraic modular forms as introduced by Gross (\cite{GrossAlgebraicModularForms}). Let $\GG$ be a connected, reductive linear algebraic group over $\MQ$\footnote{We only limit ourselves to the rationals for the purpose of this introduction while later working over arbitrary (totally real) number fields.} such that $\GG(\MR)$ is compact and choose an open, compact subgroup $K \subset \GG(\hQ)$ as well as a an irreducible $\MQ$-rational representation $V$ of $\GG$. In this notation the space of algebraic modular forms of level $K$ and weight $V$ is the space
\begin{equation}
 M(V,K)=\left\{ f:\GG(\hat{k}) \rightarrow V ~|~  \substack{f(g \gamma \kappa)=gf(\gamma) \text{ for } \gamma \in 
\GG(\hat{k}),\\~g \in \GG(k),\kappa \in K  }\right\}.
\end{equation}
This space comes equipped with an action of the Hecke algebra $H_K$ of compactly supported $K$-bi-invariant functions under convolution and our main algorithmic goal is the explicit computation of this action for given instances of $\GG$, $K$ and $V$. The first general work formulated in this language is due to Lanksy and Pollack (cf. \cite{LanskyDecomposition,LanskyPollack}) who worked in a very general setup (which will be of considerable use to us later on) and then performed explicit calculations for compact forms of $G_2$ and $\mathrm{PGSp}_4$ over the rationals. Cunningham and Demb\'el\'e computed Siegel modular forms, i.e. algebraic modular forms for compact forms of $\mathrm{GSp}_4$ over totally real fields of narrow class number one (cf. \cite{CunninghamDembeleGenus2}) and Loeffler performed computations for unitary groups (of degree $2$ and $3$) over imaginary quadratic number fields (cf. \cite{LoefflerExplicitComputations}). In addition Greenberg and Voight introduced a general framework of lattice methods for algebraic modular forms on classical groups, in particular unitary and orthogonal groups (cf. \cite{GreenbergVoightLatticeMethods}).

Here we describe a method to compute two Hecke operators (acting on distinct spaces of algebraic modular forms) simultaneously by employing Eichler's idea. To that end we introduce for two open compact subgroups $K_1,K_2$ of $\GG(\hQ)$ the so called intertwining operator $T(K_1,K_2):M(V,K_1) \rightarrow M(V,K_2)$. It turns out that $T(K_2,K_1)$ can easily be obtained from $T(K_1,K_2)$ and that $T(K_2,K_1)T(K_1,K_2)$ and $T(K_1,K_2)T(K_2,K_1)$ act as elements of $H_{K_1}$ and $H_{K_2}$, respectively, whose action can consequently be obtained by only computing $T(K_1,K_2)$. If $K_1$ and $K_2$ only differ at a single prime $p$ where they are two parahoric subgroups containing a common Iwahori subgroup it is possible to determine the coefficients of $T(K_2,K_1)T(K_1,K_2)$ (which we call the Eichler element of $K_1$ and $K_2$) in the standard basis of $H_{K_1}$. This allows us to study how many of the generators of $H_{K_1}$ can be computed in this way and it turns out that for $\GG$ of type $C_n$ we can obtain the action of the full (local) Hecke algebra. Finally we apply our method to the computation of algebraic modular forms for compact forms of symplectic groups (over totally real number fields) with respect to a parahoric level structure defined by a lattice, making it possible to hand off most of the computation to the Plesken-Souvignier algorithm (\cite{PleskenSouvignier}) for isometry testing.  

This article is organized as follows. We start by reviewing the basic terminology regarding algebraic modular forms and Hecke algebras (section \ref{heckealgebrassection}). In section \ref{doublecosetssection} we review the structure of double cosets in (split semisimple) $p$-adic groups. In section \ref{eichlersection} we introduce the concept of intertwining operators and Eichler elements and study their properties. Finally we present some explicit computational examples in section \ref{resultssection}.

\section{Algebraic modular forms and Hecke operators}\label{heckealgebrassection}
Here we want to review the basic notions in the theory of algebraic modular forms. The primary reference is Gross's original article \cite{GrossAlgebraicModularForms}.
\subsection{Algebraic modular forms}
Let $k$ be a totally real number field with ring of integers $\CO_k$ and let
\begin{equation*} 
 k_\infty:=\MR \otimes_\MQ k \cong \MR^{[k:\MQ]}.
\end{equation*}
Considering the finite places we set $\hat{\MQ}$ the finite adeles of $\MQ$ (i.e. the elements of $\prod_{p \text{ 
prime}}\MQ_p$ which are integral at all but finitely many places) and $\hat{k}:=k \otimes_\MQ \hat{\MQ}$ the finite 
adeles of $k$. We identify $\hat{k}$ with the set
\begin{equation}
 \hat{k}=\left\{(x_\fp)_\fp \in \prod_{\fp\subset \CO_k\text{ prime}}k_\fp~|~x_\fp \in \CO_\fp \text{ f.a.a. } \fp \right\}
\end{equation}
  and we denote the (full) ring of adeles of $k$ by $\MA_k:=k_\infty \times \hk$. 

Let $\GG$ be a connected reductive linear algebraic group over $k$ such that $\GG(k_\infty)$ is compact (had we 
not 
already stipulated $k$ to be totally real we would get it as a consequence here) and let $\rho:\GG \rightarrow 
\GL_V$ be an irreducible finite-dimensional rational representation of $\GG$ defined over some extension of $k$.

\begin{Definition}
 Let $K$ be an open compact subgroup of $\GG(\hk)$. The space of algebraic modular forms of weight $V$ and level $K$ is defined as
\begin{equation}
\begin{split}
M(V,K)&=\left\{ f:\GG(\hat{k})/ K \rightarrow V ~|~  \substack{f(g\gamma )=gf(\gamma) \text{ for } \\ \gamma \in \GG(\hat{k}),~g \in \GG(k) }\right\} \\
&\cong \left\{ f:\GG(\hat{k}) \rightarrow V ~|~  \substack{f(g \gamma \kappa)=gf(\gamma) \text{ for } \gamma \in 
\GG(\hat{k}),\\~g \in \GG(k),\kappa \in K  }\right\}.
\end{split}
\end{equation}
\end{Definition}

Let now $K\subset \GG(\hk)$ be an open compact subgroup. The structure of $M(V,K)$ is summarized in the following proposition.
\begin{Proposition}[\protect{\cite[Prop. (4.3),(4.5)]{GrossAlgebraicModularForms}}]
 Set $\Sigma_K:=\GG(k) \backslash \GG(\hat{k}) / K$. The following holds:
\begin{enumerate}
 \item The set $\Sigma_K$ is finite.
 \item If $\alpha_i,1\leq i \leq h,$ is a system of representatives for $\Sigma_K$ and 
\begin{equation}
 \Gamma_i:=\GG(k) \cap \alpha_i K \alpha_i^{-1},
\end{equation}
then
\begin{equation}
 M(V,K) \rightarrow \bigoplus_{i=1}^h V^{\Gamma_i},~f\mapsto (f(\alpha_1),...,f(\alpha_h))
\end{equation}
is an isomorphism of vector spaces, where $V^{\Gamma_i}$ denotes the $\Gamma_i$ fixed points in $V$. In particular $M(V,K)$ is finite-dimensional.
\end{enumerate}
\end{Proposition}
Note that the groups $\Gamma_i$ are discrete subgroups of the compact group $\GG(k_\infty)$ hence finite. Moreover since $\GG(k_\infty)$ is compact the space $V$ carries a $\GG(k)$-invariant (totally positive) inner product, $\langle-,-\rangle$, which we can use to define a Peterson scalar product on the space $M(V,K)$. As before let $\alpha_i,1\leq i \leq h,$ be a system of representatives for $\Sigma_K$ and $\Gamma_i=\GG(k) \cap \alpha_i K \alpha_i^{-1}$. For $f,f' \in M(V,K)$ we define
\begin{equation}\label{InnerProduct}
 \langle f,f' \rangle_M :=\sum_{i=1}^h \frac{1}{|\Gamma_i|}\langle f(\alpha_i),f'(\alpha_i)\rangle.
\end{equation}
The so defined map $\langle-,-\rangle_M$ is obviously a totally positive definite symmetric bilinear from on $M(V,K)$ and does not depend on the choice of our representatives $\alpha_i$.

\subsection{Hecke operators}
We keep the notation from the previous subsection.

The space of algebraic modular forms comes equipped with the action of the Hecke algebra of $\GG$ with respect to $K$.
\begin{Definition}
 The Hecke algebra $H_K=H(\GG,K)$ is the ($\MQ$-)algebra of all locally constant, compactly supported functions $\GG(\hat{k}) \rightarrow \MQ$ which are $K$-bi-invariant. The multiplication in $H_K$ is given by convolution with respect to the (unique) Haar measure $d\lambda_K$ giving the compact group $K$ measure $1$, i.e.
\begin{equation}
 (F \cdot F')(\gamma)=\int_{\GG(\hat{k})}F(x)F'(x^{-1}\gamma)d\lambda_K(x)=\int_{\GG(\hat{k})}F(\gamma y^{-1})F'(y)d\lambda_K(y)
\end{equation}
for $F,F' \in H_K$ and $\gamma \in \GG(\hat{k})$.
\end{Definition}

The algebra $H_K$ has a canonical basis given by the characteristic functions of the double cosets with respect to $K$, $\ind_{K\gamma K}$, $\gamma \in \GG(\hk)$. The action of $\ind_{K\gamma K}\in H_K$ on $M(V,K)$ is given as follows: Decompose $K\gamma K= \sqcup_{i} \gamma_i K$, then $\ind_{K \gamma K}$ acts via the operator $T(\gamma)=T(K\gamma K) \in \End(M(V,K))$ defined by
\begin{equation}
(T(\gamma)f)(x)=\sum_i f(x\gamma_i) \text{ for } f \in M(V,K),x\in \GG(\hk).
\end{equation}

The additive extension of $T$ to $H_K$ is a homomorphism of $\MQ$-algebras and the action of $H_K$ on $M(V,K)$ is compatible with the inner product on $M(V,K)$ in the following sense.
\begin{Proposition}[\protect{\cite[Prop. (6.9)]{GrossAlgebraicModularForms}}]\label{AdjointOperator}
 The adjoint operator of $T(\gamma)$ is given by $T(\gamma^{-1})$ (as an element of $\End(M(V,K))$).
\end{Proposition}
In particular this means that $M(V,K)$ is a semisimple $H_K$-module.

\subsection{Open compact subgroups arising from lattices}

A classical way in which open, compact subgroups of $\GG(\hk)$ arise is as stabilizers of lattices. To that end let $\GG \hookrightarrow \GL_W$ be a faithful $k$-rational representation of $\GG$ and $L \subset W$ a (full) $\CO_k$-lattice in $W$. The group $\GL_W(\hk)$ (and thus also $\GG(\hk)$) acts on the set of lattices in $W$ and we obtain an open, compact subgroup $K_L=\stab_{\GG(\hk)}(L)$ with
\begin{equation}
 K_L=\prod_{\fp \text{ prime}}K_{L,\fp},\text{ where } K_{L,\fp}=\stab_{\GG(k_\fp)}(L \otimes \CO_{\fp}).
\end{equation}
Moreover $K_{L,\fp}$ is a hyperspecial maximal compact subgroup for all but finitely many finite primes of $\CO_k$ (cf. \cite[Prop. 3.3]{CohenNebePlesken}).

In this situation the task of decomposing $\GG(\hk)$ into $\GG(k)$-$K_L$-double cosets becomes the task of finding representatives for the isomorphism classes in the ($\GG$-)genus of $L$, i.e. decomposing the $\GG(\hk)$-orbit of $L$ into $\GG(k)$-orbits. The class number $|\Sigma_{K_L}|$ is then also called the class number of $L$ and the complexity of $\Sigma_{K_L}$ is in some sense measured by the mass of $L$,
\begin{equation}
 \mass(L):=\mass_\GG(L):=\sum_{i=1}^h \frac{1}{|\Gamma_i|},
\end{equation}
where $\Gamma_i=K_L \cap \alpha_i K_L \alpha_i^{-1}$ and $\GG(\hk)=\bigsqcup_{i=1}^h \GG(k)\alpha_iK_L$, which means that $\{L_i=\alpha_i L|1\leq i \leq h\}$ is a system of representatives for the genus of $L$.

The mass of $L$ depends only on local information and can be computed without writing down a system of representatives for the genus. Formulas to do so are readily available in the literature (see for example \cite{GanYuGroupSchemes} for the case of classical groups and \cite{CohenNebePlesken} for semisimple groups split at every prime). These formulas rely on the fact that it is possible to compare the masses of two lattices, an idea that first appeared in the work of Eichler (cf. \cite[Satz 8]{EichlerAehnlichkeitsklassen}) and works as follows. Let $L_1,L_2$ be two lattices in faithful $\GG$-modules (not necessarily the same one) and $K_1,K_2 \leq \GG(\hk)$ the associated open and compact subgroups. Then
\begin{equation}
 \mass(L_1)\cdot [K_1:(K_1 \cap K_2)]=\mass(L_2)\cdot [K_2:(K_1 \cap K_2)].
\end{equation}
Note that the quantities $[K_1:(K_1 \cap K_2)]$ and $[K_2:(K_1 \cap K_2)]$ are indeed finite since $K_1 \cap K_2$ is again open and compact and thus of finite index in both $K_1$ and $K_2$.

\section{Double cosets in $p$-adic groups}\label{doublecosetssection}
We fix the following notation which is essentially identical to that used in \cite{LanskyPollack} and \cite{LanskyDecomposition} to ensure compatibility.

Let $F$ be a local field of characteristic $0$ with ring of integers $\CO_F$, uniformizer $\pi$ and finite residue class field of order $q$ and characteristic $p$.\\ 
Furthermore let $\GG$ be a connected, semisimple, linear algebraic group defined and split over $F$. 
There is a Chevalley group scheme $\underline{G}$ over $\CO_F$ such that $K:=\underline{G}(\CO_F) \leq 
\underline{G}(F)=\GG(F)$ is a hyperspecial maximal compact subgroup and such that the special fiber 
$\underline{G}_{\CO_F/\pi\CO_F}$ is again semisimple of the same type as $\GG$.

Let $\underline{T} \leq \underline{G}$ be a split maximal torus scheme (whose generic fiber $\underline{T}_{F}$ 
we call $\BT$). We set $\mathbb{N}_{\BT}$ the normalizer of $\BT$ in $\GG$, i.e. 
$\mathbb{N}_\BT(A):=N_{\GG(A)}(\BT(A))$ for all $A \in \Algcat{F}$. Furthermore let $X^*(\BT)=\Hom(\BT,\GG_m)$ and
 $X_*(\GG_m,\BT)$ be the character and cocharacter module of $\BT$, respectively. Let $\Phi \subset 
X^*(\BT)$ be the (finite) set of roots (i.e. the non-trivial weights occurring in the adjoint representation). We 
choose some positive subset $\Phi^+\subset \Phi$ (or equivalently a Borel subgroup $\BT(F) \subset B \subset 
\GG(F)$) and denote by $\Delta$ the corresponding simple (or indecomposable) roots. Dually to this, let $\Phi^\vee\subset 
X_*(\BT)$ be the set of coroots and $\alpha \mapsto \alpha^\vee$ the usual bijective correspondence. 

Given $\alpha \in \Phi$ we denote by $x_\alpha: \GG_a \rightarrow \underline{U_\alpha}$ the isomorphism between 
$\GG_a$ and the one-dimensional unipotent subgroup scheme $\underline{U_\alpha} \leq \underline{G}$ (whose 
generic fiber $(\underline{U_\alpha})_{F}$ we will call $\BU_\alpha$). The morphism $x_\alpha$ when considered as 
a map $F \rightarrow \BU_\alpha(F)$ restricts to $\CO_F$ with $\underline{U_\alpha}(\CO_F)=\BU_\alpha(F) \cap 
K$.

The (finite) Weyl group of $\GG$, defined as $\mathbb{N}_\BT(F)/\BT(F)=(\mathbb{N}_\BT(F) \cap 
K)/\underline{T}(\CO_F)$, will here be denoted by $W_0$ (to avoid ambiguity), while we use the symbol $\tilde{W}$ 
for the extended affine Weyl group $\mathbb{N}_\BT(F)/\underline{T}(\CO_F)$. Then both $W_0$ and $\tilde{W}$ 
act as groups of affine transformations on the vector space $X_*(\BT) \otimes_\MZ \MR$ and $W_0$ is precisely 
the stabilizer of $0 \in X_*(\BT)\otimes_\MZ \MR$ in $\tilde{W}$. Furthermore there is an isomorphism
\begin{equation}
 \tilde{W} \cong X_*(\BT) \rtimes W_0
\end{equation}
where we embed $X_*(\BT)$ into $\tilde{W}$ as a normal subgroup of translations (acting in the obvious way on 
$X_*(\BT) \otimes_\MZ \MR$). In this sense set $t_\lambda=t(\lambda) \in \tilde{W}$ the translation corresponding to $\lambda \in 
X_*(\BT)$ which yields the following identity:
\begin{equation}
 w^{-1}t_\lambda w = t_{\lambda w},
\end{equation}
for $w \in W_0$ and $\lambda \in X_*(\BT)$.

 The Weyl group $W_0$ is a finite Coxeter 
group with set of involutive generators $S_0=\{w_\alpha~|~\alpha \in \Delta\}$, where $w_\alpha$ simply denotes 
the reflection through the vanishing hyperplane of the root $\alpha \in \Phi$. Now decompose 
$\Phi=\Phi_1\cup...\cup\Phi_m$ into irreducible root systems with corresponding simple systems $\Delta_i,~1\leq 
i \leq m$ (such that each $\Phi_i$ is the root system of an almost simple component of $\GG$. If we put 
$\alpha_{0,i}$ the (unique) highest root of $\Phi_i$ (with respect to the simple system $\Delta_i$) we can form 
a larger Coxeter group with generators $\tilde{S}:=S_0 \cup \{t_{\alpha_{0,i}^\vee}w_{\alpha_{0,i}}\}$ which is 
isomorphic to the affine Weyl group $W_{af}$ associated to $\Phi$, which we will now think of as a subgroup of 
$\tilde{W}$ via this isomorphism.

We make the canonical choice for an Iwahori subgroup of $G$ by letting $I$ equal the subgroup generated by 
$\UT(\CO_F)$, the groups $x_\alpha(\CO_F)=\underline{U_\alpha}(\CO_F)$ for $\alpha \in \Phi^+$ and the 
groups $x_\alpha(\pi \CO_F)$ for $\alpha \in \Phi^- =-\Phi^+$. This is the canonical choice for $I$ since it is just 
the inverse image of the Borel subgroup associated to $\Phi^+$ under the reduction modulo $\pi$ from $K$ to the 
special fiber of $\underline{G}$ (cf. \cite{TitsReductiveGroups}). Under these definitions the triple 
$(\GG(F),I,\mathbb{N}_\BT(F))$ is a generalized Tits system.

The group $\tilde{W}$ is an extension of $W_{a}$ by a group $\Omega$ which one can find as follows: Put $\tilde{I}$ the normalizer of $I$ in $\GG(F)$ then
\begin{equation}
 \Omega=(\mathbb{N}_\BT(F) \cap \tilde{I})/\underline{T}(\CO_F) \subset \tilde{W}.
\end{equation}
The group $\Omega$ is finite Abelian and canonically isomorphic to $X_*(\BT)/\Lambda$, where $\Lambda \leq_\MZ 
X_*(\BT)$ is the lattice generated by the coroots $\Phi^\vee$ (in 
particular $\Omega$ is trivial if $\GG$ is simply connected). It normalizes $W_{af}$ and we get a split 
extension
\begin{equation}
 \tilde{W} \cong W_{af} \rtimes \Omega.
\end{equation}
As usual we can consider the length function $w \mapsto \ell(w)$ on the Coxeter group $W_{af}$ (with respect to the generating system $\tilde{S}$). This length function extends to $\tilde{W}$ by setting $\ell(\rho w):=\ell(w\rho):=\ell(w)$ for $w \in W_{af}$ and $\rho \in \Omega$. In this sense we will call an expression $w=w_1...w_r\rho$ with $w_i \in \tilde{S},~1\leq i \leq r$, and $\rho \in \Omega$ reduced if $\ell(w)=r$.
\subsection{Double cosets}
Let $W_1$ be a subgroup of $W_{af}$ with $W_1=\langle S \rangle$ where $S=W_1 \cap \tilde{S}$. Remember that such a subgroup is called a special (or standard parabolic) subgroup and is a Coxeter group in its own right where the length function of $W_1$ is just the restriction of the length function of $W_{af}$ to $W_1$. If $W_2 \leq W_1$ is another special subgroup (generated by $S'=W_2 \cap \tilde{S}$) we set
\begin{equation}
 [W_1 / W_2]:=\left\{ w \in W_1~:~\ell(ww')=\ell(w)+\ell(w') \text{ for all } w' \in W_2\right\}.
\end{equation}
Note that the elements of $[W_1 / W_2]$ are just the representatives of $W_1 / W_2$ of minimal length (cf. \cite[\textsection 2.5]{CarterSimpleGroupsLieType}).

For the remainder of this section we will consider two special subgroups $W_1,W_2$ of $W_{af}$ with respective 
sets of generators $S_1$ and $S_2$ and intersection $W_{1,2}:=W_1 \cap W_2$ (another special subgroup generated 
by $S_1 \cap S_2$). Furthermore define for $\sigma \in \tilde{W}$ the group $W_1^{\sigma W_2}$ as $W_1 \cap 
\sigma W_2 \sigma^{-1}$ which is just the stabilizer of the coset $\sigma W_2$ in $W_1$. Finally we will choose a 
system $[W_1\backslash \tilde{W} / W_2]\subset \tilde{W}$ of representatives for  $W_1\backslash \tilde{W} / 
W_2$ of minimal length, i.e. each $\sigma \in [W_1\backslash \tilde{W} / W_2]$ is of minimal length in 
$W_1\sigma W_2$ (a priori there is no reason to assume that these are unique). 

One of the main results of \cite{LanskyDecomposition} is the fact that for $\sigma \in [W_1\backslash \tilde{W} 
/ W_2]$, the group $W_1^{\sigma W_1}$ is a special subgroup of $W_1$. Here we want to present a slight generalization of this fact.

We start by restating the results from \cite{LanskyDecomposition} the first being the following analogue of the deletion condition for Coxeter groups.
\begin{Proposition}[\protect{\cite[Prop. 4.2]{LanskyDecomposition}}]
 Let $w \in \tilde{W}$ with reduced expression $w=\rho s_1...s_r$ (so $\rho \in \Omega$, $s_i \in \tilde{S}$ and $r=\ell(w)$). Then for all $s \in \tilde{S}$ exactly one of the following holds:
\begin{enumerate}
 \item $\ell(ws)=\ell(w)+1$.
 \item There is an $1 \leq i \leq r$ such that $w=\rho s_1...\hat{s_i}...s_rs$ (in which case $\ell(ws)=\ell(w)-1$).
\end{enumerate}
\end{Proposition}

\begin{Lemma}[\protect{\cite[Lemma 4.3]{LanskyDecomposition}}]\label{Lansky4.3}
 Let $w,w' \in \tilde{W}$ with $\ell(w'w)=\ell(w')+\ell(w)$ and $s \in \tilde{S}$ with $\ell(ws)=\ell(w)+1$. Then exactly one of the following holds:
\begin{enumerate}
 \item $\ell(w'ws)=\ell(w'w)+1$.
 \item $w'w=\hat{w}'ws$ for some $\hat{w}' \in \tilde{W}$ with $\ell(\hat{w}')<\ell(w')$.
\end{enumerate}
Furthermore if $w'=\rho s_1...s_r$ is a minimal expression then $\hat{w}'=\rho s_1...\hat{s_i}...s_r$ for some $1\leq i \leq r$. This implies that if $w'$ is an element of some special subgroup of $\tilde{W}$, then $\hat{w}'$ is an element of that subgroup, too.
\end{Lemma}

\begin{Proposition}[\protect{\cite[Prop. 4.5]{LanskyDecomposition}}]
For $\sigma \in [W_1\backslash \tilde{W} / W_2]$ the group
\begin{equation}
W_1^{\sigma W_2}=W_1 \cap {^\sigma W_2}
\end{equation}
 is a special subgroup of $W_1$.
\end{Proposition}

In particular this means that the expression $[ W_1 / W_1^{\sigma W_2}]$ is well-defined and we have the following length additivity property:
\begin{Theorem}[\protect{\cite[Thm. 4.6]{LanskyDecomposition}}]\label{Lansky4.6}
 Fix two elements $\sigma \in [W_1\backslash \tilde{W} / W_2]$ and $\tau \in [ W_1 / W_1^{\sigma W_2}]$ then
 \begin{equation}
  \ell(\tau \sigma w)=\ell(\tau)+\ell(\sigma)+\ell(w)
 \end{equation}
for all $w \in W_2$.
\end{Theorem}

\begin{Corollary}[\protect{\cite[Cor. 4.7]{LanskyDecomposition}}]\label{Lansky4.7}
 Let $\sigma \in [W_1\backslash \tilde{W} / W_2]$ then $\sigma$ is the unique element of minimal length in $W_1\sigma W_2$.
\end{Corollary}

We now want to study what happens if we consider double cosets with respect to the intersection $W_{1,2}$ and how these correspond to the double cosets with respect to $W_1$ or $W_2$. To that end let $W_1'=W_1\Omega_1$ with $\Omega_1 \leq \Omega$ fixing $S_1$. Then $W_{1,2}=W_1 \cap W_2 = W_1' \cap W_2$. 

\begin{Lemma}\label{CosetSwitch}
  We have $\ell(\sigma w_2)=\ell(w_2 \sigma)=\ell(w_2)+\ell(\sigma)$ for all $w_2 \in W_2$ and all $\sigma \in [W_{1,2}\backslash W_1' / W_{1,2}]$, where $[W_{1,2}\backslash W_1' / W_{1,2}]=[W_{1,2}\backslash \tilde{W} / W_{1,2}] \cap W_1'$ corresponds to the double cosets with respect to $W_{1,2}$ that are contained in $W_1'$.
\end{Lemma}
\begin{proof}
 We will show this by induction on $\ell(w_2)$ where the assertion is trivial for $\ell(w_2)=0$. Let $w_2=w_2's$ with $s \in S_2$.
Then $\ell(\sigma w_2')=\ell(\sigma)+\ell(w_2')$ by induction and via Lemma \ref{Lansky4.3} we have either $\ell(\sigma w_2's)=\ell(\sigma w_2')+1$ (which we want to show) or $\sigma w_2' = \sigma' w_2's$ with $\ell(\sigma')<\ell(\sigma)$ and $\sigma' \in W_1'$. But in the latter case we have $\sigma=\sigma'w_2'sw_2'^{-1}$ and $w_2'sw_2'^{-1}=\sigma'^{-1}\sigma \in W_2 \cap W_1'$ which contradicts the assumption that $\sigma$ is of minimal length in its double coset with respect to $W_{1,2}$.

The other equality follows analogously.
\end{proof}

Using this length additivity we can show that the representatives of minimal length with respect to $W_{1,2}$ are already of minimal length with respect to the larger group $W_2$ as long as they are contained in $W_1'$.
\begin{Lemma}\label{DoubleCosetEmbedding}
 The following holds: $\sigma \in [W_{1,2}\backslash W_1' / W_{1,2}]$ already implies $\sigma \in [W_{2}\backslash \tilde{W} / W_{2}]$.
\end{Lemma}
\begin{proof}
 We need to show that $\ell(w_2\sigma w_2')\geq \ell(\sigma)$ for all $w_2,w_2' \in W_2$ and will do so by induction on $\mathrm{min}(\ell(w_2),\ell(w_2'))$. If this minimum is $0$ the assertion follows from the length additivity property in Lemma \ref{CosetSwitch} so let $\mathrm{min}(\ell(w_2),\ell(w_2'))>0$ be realized (without loss of generality) at $w_2$. Then we can write $w_2=s\tilde{w_2}$ with $s \in S_1$ and $\ell(\tilde{w_2})=\ell(w_2)-1$. Assume $\ell(w_2\sigma w_2')<\ell(\sigma)$ while (by induction) $\ell(\tilde{w_2}\sigma w_2')\geq \ell(\sigma)$. We see
\begin{equation}
 \ell(\sigma)>\ell(w_2\sigma w_2')=\ell(s\tilde{w_2}\sigma w_2')\geq \ell(\tilde{w_2} \sigma w_2')-1 \geq \ell(\sigma)-1
\end{equation}
and hence $\ell(\tilde{w_2}\sigma w_2')=\ell(\sigma)$. Thus
\begin{equation}
 \ell(\sigma)=\ell(\tilde{w_2}\sigma w_2')\geq \ell(\sigma w_2')-\ell(\tilde{w_2})=\ell(\sigma)+\ell(w_2')-\ell(\tilde{w_2})
\end{equation}
which implies $\ell(w_2') \leq \ell(\tilde{w_2})=\ell(w_2)-1$ but we had $\ell(w_2')\geq \ell(w_2)$ which is a contradiction. Hence $\ell(w_2\sigma w_2')\geq \ell(\sigma)$ which completes the proof.
\end{proof}

Using this lemma and Theorem \ref{Lansky4.6} we easily get the following result on the intersection of $W_2$ and ${^\sigma W_2}$ with $\sigma \in [W_{1,2}\backslash W_1' / W_{1,2}]$.
\begin{Corollary}\label{SpecialIntersection}
 In the same notation as above: For $\sigma \in [W_{1,2}\backslash W_1' / W_{1,2}]$ the group $W_2 \cap {^\sigma W_2}$ is a special subgroup of $W_2$.
\end{Corollary}

Not only are the elements of $[W_{1,2}\backslash W_1' / W_{1,2}]$ of shortest length in their $W_2-W_2-$double coset but no two of them define the same one.
 
\begin{Corollary}\label{CosetRepresentativesW121}
 For $\sigma,\sigma'\in [W_{1,2}\backslash W_1' / W_{1,2}]$ we have $W_2\sigma W_2=W_2 \sigma' W_2$ if and only if $\sigma = \sigma'$.  
\end{Corollary}
\begin{proof}
 This follows from Lemma \ref{DoubleCosetEmbedding} together with the fact that the elements of $[W_{2}\backslash \tilde{W} / W_{2}]$ are the unique elements of shortest length in their respective double coset by Lemma \ref{Lansky4.7}.
\end{proof}

Note that it is not possible to generalize the last result to also include the case where $W_2$ contains a non-trivial subgroup of $\Omega$ as the following example shows.
\begin{Example}\label{CounterexampleDoubleCosets}
 Consider the extended affine Weyl group with Dynkin diagram
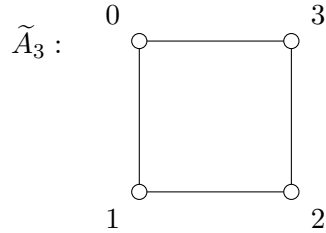
\begin{center}
\begin{tikzpicture}
    
    \draw (0,0) -- (2,0);
    \draw (0,0) -- (0,2);
    \draw (2,2) -- (2,0);
    \draw (2,2) -- (0,2);
    \draw[fill=white] (0,0) circle(.1);
    \draw[fill=white] (2,0) circle(.1);
    \draw[fill=white] (0,2) circle(.1);
    \draw[fill=white] (2,2) circle(.1);
    
    \node at (-1.35,2) {$\tilde{A}_3:$};
    \node at (-0.35,-0.35) {$1$};
    \node at (-0.35,2.35) {$0$};
    \node at (2.35,2.35) {$3$};
    \node at (2.35,-0.35) {$2$};

\end{tikzpicture} 
\end{center}
and finite group $\Omega=\langle \omega \rangle \cong{C_2}$ interchanging $s_0$ with $s_2$ and $s_1$ with $s_3$. We set $W_2':=\langle s_0,s_2,\omega\rangle$ and $W_1=W_1'=\langle s_1,s_3 \rangle$. Then $W_2'\cap W_1'=\{1\}$ and hence $[W_{1,2}\backslash W_1 /W_{1,2}]=W_1$, but $s_1$ and $s_3$ define the same $W_2'-W_2'-$coset since $s_3=\omega s_1 \omega \in W_2's_1W_2'$.
\end{Example}

\subsection{Coset decompositions}
We briefly want to supplement the previous subsection by reviewing some results on the decomposition of double cosets with respect to parahoric subgroups. However, we will only state the results we need for our following considerations, in particular we do not give the explicit formula for coset decomposition since this would require a bunch of otherwise not needed notation and rather refer the reader to \cite{LanskyDecomposition}.

Let $P_1',P_2'$ be two parahoric subgroups of $\GG(F)$ containing $I$, then $P_i'=I W_i' I$, $i=1,2$, for certain (finite) subgroups $W_i' < \TW$ and $W_i'$ decomposes as $W_i'=W_i\Omega_i$ where $W_i=W_i'\cap W_{af}=\langle W_i' \cap \tilde{S} \rangle$ is normalized by $\Omega_i=\Omega \cap W_i'$.

\begin{Proposition}[\cite{IwahoriMatsumotoBruhatDecomposition}]\label{ParahoricIndex}
 We have $[P_1':I]=|\Omega_1|\sum_{w \in W_1} q^{\ell(w)}$ and hence if $P_2' \subset P_1'$  
\begin{equation}
 [P_1':P_2']=[\Omega_1:\Omega_2] \sum_{w \in [W_1/W_2]}q^{\ell(w)}.
\end{equation}
\end{Proposition}

For computations with Hecke operators one needs to decompose double cosets of the form $P_1'xP_2'$ into $P_2'$-left cosets. An explicit formula to do this can be found in \cite[Thm. 5.7]{LanskyDecomposition}. Here we merely need the fact that this is actually algorithmically feasible as well as the following corollary.
\begin{Corollary}[\protect{\cite[Cor 5.8]{LanskyDecomposition}}]
Let $\sigma \in [W_1'\backslash \TW / W_2']$ and denote the stabilizer of $\sigma$ in $\Omega_1 \cap \Omega_2$ by $\Omega_{1,2}^\sigma$. Then we have
\begin{equation}
 |P_1'\sigma P_2'/P_2'|=[\Omega_1:\Omega_{1,2}^\sigma]\cdot q^{\ell(\sigma)}\cdot\sum_{w \in [W_1/W_1^{\sigma W_2}]}q^{\ell(w)}.
\end{equation}
\end{Corollary}

\section{Intertwining operators and Eichler algebras}\label{eichlersection}
We now want to introduce the method that allows us to compute two Hecke operators at once. Let $k$, $\GG$ and $V$ be as in Section \ref{heckealgebrassection} and let us fix two open and compact subgroups $K_1$ 
and $K_2$ of $\GG(\hk)$. We want to consider the spaces of algebraic modular forms of weight $V$ and levels 
$K_1$ and $K_2$, respectively, $M_i:=M(V,K_i),~i=1,2$.

The group $K_1 \cap K_2$ is again open and compact, hence of finite index in both $K_1$ and $K_2$. We fix coset representatives $m_i\in K_2, ~i \in I$, and $l_j\in K_1,~ j \in J$, such that
\begin{equation}
 \begin{split}
  K_1&=\bigsqcup_{j \in J} l_j(K_1 \cap K_2)=\bigsqcup_{j \in J}(K_1 \cap K_2)l_j^{-1},\\
  K_2&=\bigsqcup_{i \in I} m_i(K_1 \cap K_2)=\bigsqcup_{i \in I}(K_1 \cap K_2)m_i^{-1}.
 \end{split}
\end{equation}

The following is a simple observation but puts an easy upper bound on the amount of computations we have to do.

\begin{Remark}
 Let $\GG(\hk)=\bigsqcup_{s \in S}\GG(k)\gamma_s K_1$, then 
\begin{equation}
\GG(\hk)=\bigcup_{s \in S,j\in J}\GG(k)\gamma_sl_jK_2.
\end{equation}
 Thus the set $\{\gamma_sl_j:j \in J,s\in S\}$ contains a system of representatives for $\GG(k)\backslash \GG(\hk) /K_2$.
\end{Remark}

\subsection{Intertwining operators}

Remember that in the case where $K_1$ and $K_2$ are the stabilizers of two lattices $L$ and $M$, Eichler's 
method can be used to obtain the mass of the genus of $M$ from the mass of the genus of $L$ via
\begin{equation}
 \mass(\genus(M))=\mass(\genus(L))\cdot \frac{[K_1:K_1\cap K_2]}{[K_2:K_1 \cap K_2]}.
 \end{equation}

We want to take this idea one step further and try to employ the connection between $K_1$ and $K_2$ to compute the action of certain Hecke operators.  

\begin{Definition}
 We define 
\begin{equation}
 T^1_2=T^{K_1}_{K_2}: M_1 \rightarrow M_2,~ f\mapsto f',
\end{equation}
where
\begin{equation}
 f'(\gamma)=\sum_{i \in I} f(\gamma m_i) \text{ for all } \gamma \in \GG(\hk).
\end{equation}
We call $T^1_2$ the intertwining operator (with respect to $K_1$ and $K_2$ or from $M_1$ to $M_2$).
\end{Definition}

$T^1_2$ is well-defined and independent of the choice of the $m_i$ since $f$ is invariant under right-multiplication of the argument by elements of $K_1$ and hence, in particular, by elements of $K_1\cap K_2$. 

The operators $T^1_2$ and $T^2_1=T^{K_2}_{K_1}$ are connected in a way that becomes apparent if we consider the 
spaces $M_1$ and $M_2$ endowed with their inner products defined in Equation \ref{InnerProduct}.

 \begin{Lemma}
 The operators $T^1_2$ and $T^2_1$ are adjoint to each other. That is, for all $f \in M_1$ and $f' \in M_2$ we have 
\begin{equation}
 \langle T^1_2f,f'\rangle_{M_2} = \langle f,T^2_1f' \rangle_{M_1}.
\end{equation}
\end{Lemma}
\begin{proof}
 This is a tedious computation that works analogously to the proof for the adjoint of a Hecke operator.
\end{proof}

Since the adjoint operator of $T$ is uniquely determined by $T$, the above theorem shows that in applications it suffices to compute just one of the two. This observation is particularly useful in light of the next subsection.

\subsection{Eichler elements}

\begin{Lemma}
 The function $\sum_{i \in I,j \in J}\ind_{l_jm_iK_1}$ is an element of $H_{K_1}$.
\end{Lemma}

\begin{Definition}
 We will call the function 
\begin{equation}
\nu_{1,2}:=\nu(K_1,K_2):=\sum_{i \in I,j \in J}\ind_{l_jm_iK_1} \in H_{K_1}
\end{equation}
 the Eichler element (of $H_{K_1}$) with respect to $K_2$. 
\end{Definition}

Note that the definition of the Eichler element makes sense not only in the given global situation but also locally (e.g. one could look at the Eichler element of $K_{1,\fp}$ and $K_{2,\fp}$ as an element of $H_{K_{1,\fp}}$). Moreover if $K_1$ and $K_2$ decompose as products of local factors that coincide at all but one place the ``global'' Eichler element is just the embedding of the ``local'' Eichler element (at this place). In this sense we will sometimes want to think of the global Eichler element as a local object and vice versa if no confusion arises from this. 

Since $(T^2_1T^1_2f)(\gamma)=\sum_{i \in I,j \in J}f(\gamma l_jm_i)$ for all $f \in M_1$ and $\gamma \in \GG(\hk)$, the above lemma immediately implies the following corollary.
\begin{Corollary}
 The linear operator $T^2_1T^1_2:~M_1 \rightarrow M_1$ is self-adjoint and acts as the element $\nu_{1,2}$ of the Hecke algebra $H_{K_1}$ on the space $M_1$.
\end{Corollary}

Since $l_jm_iK_1 \subset K_1 m_i K_1$ and $m_i \in K_2$, the Eichler element $\nu_{1,2}$ is only supported on 
the double cosets $K_1 m K_1$ with $m \in K_2$. On the other hand write $m \in K_2$ as $ m_i\kappa$ for some $i 
\in I$ and $\kappa \in K_1\cap K_2$, then $K_1 m K_1=K_1 m_i K_1$, which shows that each of these cosets 
actually appears in the support of $\nu_{1,2}$. Let us denote (by slight abuse of notation) the set of double 
cosets in $\GG(\hk)\ds K_1$ which have a representative in $K_2$ by ${K_2}\ds{K_1}$. Then we have seen that 
\begin{equation}
 \nu_{1,2}=\sum_{i \in I,j \in J}\ind_{l_jm_iK_1}=\sum_{K_1 \kappa K_1 \in K_2 \ds K_1} \nu_{1,2}(\kappa) \ind_{K_1 \kappa K_1}
\end{equation}

and furthermore
\begin{equation}
\begin{split} 
\nu_{1,2}(\kappa)&=| \left\{(i,j) \in I \times J~:~ \kappa \in l_jm_iK_1 \right\}|\\
         &=\frac{| \left\{(i,j) \in I \times J~:~  l_jm_iK_1 \subset K_1\kappa K_1 \right\}|}{| K_1\kappa K_1 / K_1|}\\
         &=\frac{|J|\cdot | \left\{i \in I~:~  m_i K_1  \subset K_1\kappa K_1 \right\}|}{| K_1\kappa K_1 / K_1|}\\
         &=\frac{[K_1:K_1\cap K_2] \cdot |\left\{i \in I~:~ m_i K_1  \subset K_1\kappa K_1 \right\}|}{| K_1\kappa K_1 / K_1|}
\end{split}
\end{equation}

If the intersection of $K_1$ and $K_2$ is small, $\nu_{1,2}$ will have very large support and hence will in general not be of particular interest. However, if $K_1 \cap K_2$ has (in some sense) small index in both $K_1$ and $K_2$ the operator $\nu_{1,2}$ will be a linear combination of only a few elements of the standard basis of $H_{K_1}$, and thus $T(\nu_{1,2})$ might prove to be useful for the computation of the action of $H_{K_1}$ on $M_1$. 

Let us now assume that $K_i=\prod_\fq K_{i,\fq},~i=1,2,$ are both products of local factors $K_{i,\fq}$, where $\fq$ runs over the finite places of $k$ (this is the case for example when $K_1$ and $K_2$ arise as stabilizers of lattices). Furthermore assume that there is a finite place $\fp$ of $k$ such that $\GG$ is split at $\fp$ and $K_{1,\fp}$ and $K_{2,\fp}$ are two parahoric subgroups of $\GG(k_\fp)$ containing a common Iwahori subgroup $I$ while $K_{1,\fq}=K_{2,\fq}$ for all $\fq \neq \fp$. Let $\tilde{W}$ be the extended affine Weyl group of $\GG(k_\fp)$ (with respect to a suitable torus whose integral points are contained in $I$) and $W_1,W_2 \leq \tilde{W}$ the subgroups corresponding to $K_{1,\fp}$ and $K_{2,\fp}$ via
\begin{equation}
 K_{i,\fp}=\bigsqcup_{w \in W_i}IwI,~i=1,2.
\end{equation}
Furthermore set $W_{1,2}:=W_1 \cap W_2$. 

If $W_1 \subset W_{af}$ is a special subgroup of $W_{af}$ we have
\begin{equation}
 K_2 \ds K_1=\left\{ K_1 \sigma K_1 ~|~\sigma \in [W_{1,2}\backslash W_2 /W_{1,2}] \right\}
\end{equation}
(where we embed the $\sigma \in [W_{1,2}\backslash W_2 / W_{1,2}]$ into $\GG(\hk)$ in the usual way) and 
these cosets are pairwise distinct by Corollary \ref{CosetRepresentativesW121}.
In this case we can give an explicit formula for the values $\nu_{1,2}(\kappa)$ from above.
\begin{Theorem}\label{VenkovElementAffine}
 If $W_1 \subset W_{af}$ the following holds:
\begin{equation}
 \nu_{1,2}=\sum_{i \in I,j \in J}\ind_{l_jm_iK_1}=\sum_{\kappa \in [W_{1,2}\backslash W_2 / W_{1,2}]} [I W_1^{\kappa W_1} I: I (W_1^{\kappa W_1} \cap W_2)I] \ind_{K_1 \kappa K_1}.
\end{equation}
\end{Theorem}
\begin{proof}
 First note that $W_1^{\kappa W_1}$ is a special subgroup of $W_1$ by Corollary \ref{SpecialIntersection}, hence $IW_1^{\kappa W_1} I$ and  $I(W_1^{\kappa W_1} \cap W_2)I$ are in fact subgroups of $\GG(k_\fp)$ for $\kappa \in [W_{1,2}\backslash W_2 / W_{1,2}]$. Furthermore, since $K_1$ and $K_2$ coincide away from $\fp$, we can perform our computation locally and we set $P_i:=K_{i,\fp},~i=1,2$.

 Now let $\kappa \in [W_{1,2}\backslash W_2 / W_{1,2}]$, then we need to show that $\nu_{1,2}(\kappa)=[I W_1^{\kappa W_1} I: I (W_1^{\kappa W_1} \cap W_2)I]$. We already noticed that 
\begin{equation}
 \nu_{1,2}(\kappa)=\frac{[P_1:P_1\cap P_2] \cdot |\left\{i \in I ~:~ m_i P_1  \subset P_1\kappa P_1 \right\}|}{| P_1\kappa P_1 / P_1|}.
\end{equation}
Note that the set in the numerator makes sense since we can assume that the $m_i$ are only supported at $\fp$.

We will now make the terms appearing in this description more explicit.

First of all we have
\begin{equation}
  [P_1:P_1\cap P_2]=[IW_1I:I(W_1\cap W_2)I]=\frac{[IW_1I:I]}{[IW_{1,2}I:I]}.
 \end{equation}
By Lemma \ref{DoubleCosetEmbedding} we see that $\kappa$ is also of shortest length in the double coset $W_1\kappa W_1$, hence by \cite[Thm 5.2]{LanskyDecomposition} and Corollary \ref{ParahoricIndex}
\begin{equation}
 | P_1\kappa P_1 / P_1|=q^{\ell(\kappa)}\cdot \sum_{w \in [  W_1 / W_1^{\kappa W_1}]}q^{\ell(w)}=q^{\ell(\kappa)} [IW_1I:IW_1^{\kappa W_1}I],
\end{equation}
where $q$ is the order of the residue class field at $\fp$. 

Now consider the set $A:=\left\{ i \in I ~:~ m_i P_1  \subset P_1\kappa P_1\right\}$ whose cardinality gives the last term in the above expression. Let $x \in P_2$, then $xP_1 \subset P_1\kappa P_1$ if and only if $P_1xP_1=P_1\kappa P_1$ and by Corollary \ref{CosetRepresentativesW121} this is the case if and only if already $P_{1,2}xP_{1,2}=P_{1,2}\kappa P_{1,2}$. In particular this means that $i \in A$ if and only if $m_iP_{1,2} \subset P_{1,2}\kappa P_{1,2}$ which implies
\begin{equation}
 |A|=| P_{1,2}\kappa P_{1,2} / P_{1,2}|=q^{\ell(\kappa)}\sum_{w' \in [W_{1,2} / W_{1,2}^{\kappa W_{1,2}}]} q^{\ell(w')}=q^{\ell(\kappa)}[IW_{1,2}I:IW_{1,2}^{\kappa W_{1,2}}I].
\end{equation}
 Note that 
\begin{equation}
W_{1,2}^{\kappa W_{1,2}}=(W_1 \cap W_2) \cap {^\kappa(W_1 \cap W_2)}=W_{1,2} \cap {^\kappa W_1}
\end{equation}
since $\kappa \in W_2$ and hence ${^\kappa W_2}=W_2$.

Now we put all of this together and see
\begin{equation}
\begin{split}
 \nu_{1,2}(\kappa)&=\frac{[P_1:P_1\cap P_2] \cdot |\left\{i \in I ~:~ m_i P_1 \subset P_1\kappa P_1 \right\}|}{|P_1\kappa P_1 / P_1|}\\
&=\frac{[IW_1I:I] \cdot | P_{1,2}\kappa P_{1,2} / P_{1,2}|\cdot [IW_1^{\kappa W_1}I:I]}{q^{\ell(\kappa)}[IW_1I:I]\cdot[IW_{1,2}I:I]}\\
&=\frac{q^{\ell(\kappa)}[IW_{1,2}I:IW_{1,2}^{\kappa W_{1,2}}I] \cdot [IW_1^{\kappa W_1}I:I]}{q^{\ell(\kappa)}\cdot[IW_{1,2}I:I]}\\
&=\frac{[IW_{1,2}I:I] \cdot [IW_1^{\kappa W_1}I:I]}{[IW_{1,2}^{\kappa W_{1,2}}I:I]\cdot[IW_{1,2}I:I]}\\
&=\frac{[IW_1^{\kappa W_1}I:I]}{[IW_{1,2}^{\kappa W_{1,2}}I:I]}\\
&=[IW_1^{\kappa W_1}I:IW_{1,2}^{\kappa W_{1,2}}I].
\end{split}
\end{equation}
But this was exactly our assertion.
\end{proof}

Note that the condition $W_1 \leq W_{af}$ is necessary for the result to hold as is shown by Example 
\ref{CounterexampleDoubleCosets}. The condition is always fulfilled if $\GG$ is simply connected (in which case 
$\widetilde{W}=W_{af}$) or if $P_1=K_{1,\fp}$ is a hyperspecial maximal compact subgroup.

\begin{Example}\label{ExampleC2}
 In the situation of the above theorem let $\GG$ be simply connected of type $C_2$ (so a form of $\Sp_4$). The extended Dynkin diagram is of the form 
\begin{center}
\begin{tikzpicture}

    \draw (0,0.07) -- (4,0.07);
    \draw (0,-0.07) -- (4,-0.07);
    \draw (1.1,0) -- (0.9,0.2);
    \draw (1.1,0) -- (0.9,-0.2);
    \draw (2.9,0) -- (3.1,0.2);
    \draw (2.9,0) -- (3.1,-0.2);
    \draw[fill=white] (0,0) circle(.1);
    \draw[fill=white] (2,0) circle(.1);
    \draw[fill=white] (4,0) circle(.1);
    
    \node at (-1,0) {$\widetilde{C}_{2}:$};
    \node at (0,0.35) {$0$};
    \node at (2,0.35) {$1$};
    \node at (4,0.35) {$2$};

\end{tikzpicture} 
\end{center}
Let $W_1:=\langle s_1,s_2 \rangle,W_2:=\langle s_0,s_2 \rangle,W_3:=\langle s_0,s_1 \rangle$ and choose open compact subgroups $K_1,K_2$ and $K_3$ of $\GG(\hk)$ such that $K_{i,\fq}=K_{j,\fq}$ for all $\fq \neq \fp$ while $K_{i,\fp}=IW_i I$ for a suitably chosen Iwahori subgroup of $\GG(k_\fp)$. Then $K_{1,\fp}$ is a hyperspecial maximal compact subgroup and $H_{K_{1,\fp}}$ is generated by $\ind_{K_{1,\fp}s_0K_{1,\fp}},\ind_{K_{1,\fp}s_0s_1s_0K_{1,\fp}}$.

 We set $q:=\mathrm{N}_{k/\MQ}(\fp)$ and  compute:
\begin{equation}
\begin{split}
  &[W_{1,2}\backslash W_2 / W_{1,2}]=\{1,s_0\},  {^{s_0}W_1} \cap W_1 = \langle s_2 \rangle = {^{s_0}W_1} \cap W_1 \cap W_2.\\
  &\rightsquigarrow \nu_{1,2}=(q^3+q^2+q+1)\ind_{K_1}+\ind_{K_1s_0K_1}.\\
  &[W_{1,3}\backslash W_3 / W_{1,3}]=\{1,s_0,s_0s_1s_0\},{^{s_0s_1s_0}W_1} \cap W_1 = \langle s_1 \rangle = {^{s_0s_1s_0}W_1} \cap W_1 \cap W_3.\\
 &\rightsquigarrow \nu_{1,3}=(q^3+q^2+q+1)\ind_{K_1}+(q+1)\ind_{K_1s_0K_1}+\ind_{K_1s_0s_1s_0K_1}.
\end{split}
\end{equation}
Moreover 
\begin{equation}
\begin{split}
\nu_{2,1}&=(q+1)\ind_{K_2}+\ind_{K_2s_1K_2}+\ind_{K_2s_1s_2s_1K_2} \text{ and }\\ \nu_{3,1}&=(q^3+q^2+q+1)\ind_{K_3}+(q+1)\ind_{K_3s_2K_3}+\ind_{K_3s_2s_1s_2K_3}.
\end{split}
\end{equation}
This means that, starting from a decomposition $\GG(\hk)=\sqcup_{i=1}^r \GG(k) \gamma_i K_1$, we can compute the operators $T^2_1$ and $T^3_1$ to obtain the full action of the Hecke algebra $H_{K_{1,\fp}}$ on $M(V,K_1)$. In addition we achieve coset decompositions for $\GG(k) \backslash \GG(\hk) / K_2$ and $\GG(k) \backslash \GG(\hk) / K_3$ as well as two more operators (acting on $M(V,K_2)$ and $M(V,K_3)$ respectively) along the way. If we were to compute $T(K_1s_0K_1)$ and $T(K_1s_0s_1s_0K_1)$ directly we would have to deal with $q(q^3+q^2+q+1)$ and $q^3(q^3+q^2+q+1)$ left cosets respectively while computing $T^2_1$ and $T^3_1$ only requires us to consider $(q^3+q^2+q+1)$ left cosets each (where $q$ denotes the norm of $\fp$ i.e. the order of the residue class field at $\fp$).
\end{Example}

Now we want to consider the situation where $W_1 \nleq W_{af}$. To that end we change the notation slightly and consider two parahoric subgroups $P_1'=K_{1,\fp}',P_2'=K_{2,\fp}'$ with $P_i'=IW_i'I=IW_i\Omega_iI$, where $W_i \leq W_{af}$ and $\Omega_i \leq \Omega$ fixes $S_i=S \cap W_i$. Furthermore we set $P_{1,2}=P_1 \cap P_2=I(W_1 \cap W_2)I=I(W_{1,2})I$ and $P_{1,2}'=P_1' \cap P_2'=IW_{1,2}\Omega_{1,2}I$. Finally we choose representatives for certain coset decompositions as follows
\begin{equation}
\begin{split}
 P_1'&=\bigsqcup_j P_{1,2}'l_j',\\
 P_1&=\bigsqcup_j P_{1,2}l_j,\\
 P_2'&=\bigsqcup_i P_{1,2}'m_i',\\
 P_2&=\bigsqcup_i P_{1,2}m_i.\\
\end{split}
\end{equation}

In addition we choose representatives $\rho$ for $\Omega_1$, $\sigma$ for $\Omega_{1,2}$.

Now consider the element
\begin{equation}
 \tilde{\nu_{1,2}}:=\sum_{\sigma,\sigma' \in \Omega_{1,2}}\sum_{i,j}\ind_{l_j'\sigma'm_i'\sigma P_1'} \in H_{P_1'}.
\end{equation}
We see that
\begin{equation}
 \begin{split}
  \tilde{\nu_{1,2}}&=|\Omega_{1,2}|\sum_{\sigma' \in \Omega_{1,2}}\sum_{i,j}\ind_{l_j'\sigma'm_i'P_1'} \in H_{P_1'}\\
  &=|\Omega_{1,2}|^2\sum_{i,j}\ind_{l_j'm_i'P_1' } \in H_{P_1'}\\
 &=|\Omega_{1,2}|^2 \nu(P_1',P_2').
 \end{split}
\end{equation}
In particular, if we want to study $\nu(P_1',P_2')$ we can study $\tilde{\nu_{1,2}}$ instead.
\begin{Theorem}
 For $\kappa \in [W_{1,2} \backslash W_2' / W_{1,2}]$ set 
\begin{equation}
t_\kappa:=[I W_1^{\kappa W_1} I: I (W_1^{\kappa W_1} \cap W_2)I],
\end{equation}
 the coefficient of $\ind_{P_1\kappa P_1}$ in $\nu(P_1,P_2')$. Then the following holds:
\begin{equation}
 \nu(P_1',P_2')=|\Omega_{1,2}|^{-2}\sum_{\kappa \in [W_{1,2} \backslash W_2' / W_{1,2}]} t_\kappa |\Omega_1^\kappa| \ind_{P_1'\kappa P_1'},
\end{equation}
where $\Omega_1^\kappa$ denotes the stabilizer of $\kappa$ in $\Omega_1$.
\end{Theorem}
\begin{proof}
As noted before we will compute $\tilde{\nu_{1,2}}$. Notice that $\{l_j'\sigma P_{1,2}~|~\sigma \in \Omega_{1,2},j\}=\{\rho l_jP_{1,2} ~|~\rho \in \Omega_{1},j\}$ and compute for arbitrary $x$:
\begin{equation}
 \begin{split}
   \tilde{\nu_{1,2}}(x)&=\sum_{\sigma\in \Omega_{1,2},\rho\in \Omega_1}\sum_{i,j}\ind_{\rho l_jm_i'\sigma P_1'}(x) \\
  &=\sum_{\sigma\in \Omega_{1,2},\rho,\rho'\in \Omega_1}\sum_{i,j}\ind_{\rho l_j m_i'\sigma\rho'P_1}(x)\\
  &=\sum_{\sigma\in \Omega_{1,2},\rho,\rho'\in \Omega_1}\sum_{i,j}\ind_{\rho l_j m_i' \sigma P_1\rho'}(x)\\
  &=\sum_{\sigma\in \Omega_{1,2},\rho,\rho'\in \Omega_1}\sum_{i,j}\ind_{l_j m_i' \sigma P_1}(\rho x \rho')\\
  &=\sum_{\rho,\rho'\in \Omega_1}\nu(P_1,P_2')(\rho x \rho') \\
  &=\sum_{\kappa \in [W_{1,2} \backslash W_2' / W_{1,2}]} \sum_{\rho,\rho' \in \Omega_1}t_\kappa \ind_{P_1\kappa P_1}(\rho x \rho').
 \end{split}
\end{equation}
Now we decompose $P_1 \kappa P_1=\bigsqcup_r \kappa_r P_1$ and see
\begin{equation}
 \begin{split}
  \sum_{\rho,\rho' \in \Omega_1} \ind_{P_1\kappa P_1}(\rho x \rho')&=\sum_{\rho,\rho' \in \Omega_1}\sum_{r} \ind_{\kappa_rP_1\rho' }(\rho x) \\
 &=\sum_r\sum_{\rho\in \Omega_1} \ind_{\rho \kappa_r P_1' }(x)\\
 &=|\Omega_1^\kappa|\ind_{P_1' \kappa P_1'}(x). 
 \end{split}
\end{equation}
The last equality holds due to \cite[La. 5.7]{LanskyDecomposition}. Putting all of this together we achieve the result.
\end{proof}
It is noteworthy that the double cosets appearing in the sum in the last theorem are no longer necessarily distinct; in fact two cosets $P_1'\kappa P_1'$ and $P_1' \kappa'  P_1'$ with $\kappa,\kappa' \in  [W_{1,2} \backslash W_2' / W_{1,2}]$ coincide if and only if $\kappa =\sigma \kappa' \sigma'$ for suitable $\sigma,\sigma' \in \Omega_1$.

\subsection{The Eichler algebra}
In Example \ref{ExampleC2} we saw that it is possible to obtain the whole (local) Hecke algebra of $\GG$ by only computing intertwining operators in the case where $\GG$ is of type $C_2$. We now want to study in which cases we can expect this to happen and what operators we can still compute this way if we do not obtain the full Hecke algebra. To that end let again $K$ be an open compact subgroup of $\GG(\hk)$.
\begin{Definition}
 Let $\GG$ be split at the finite prime $\fp$ and $K_\fp$ be a hyperspecial maximal compact subgroup of $\GG(k_\fp)$. The Eichler algebra (at $\fp$) of $K$ is the subalgebra of $H_K$ generated by the Eichler elements $\nu(K,K')$ where $K'$ runs over the open compact subgroups of $\GG(\hk)$ such that $K_\fq'=K_\fq$ for all $\fq \neq \fp$ and $K_\fp'$ is a maximal parahoric subgroup of $\GG(k_\fp)$ that contains a common Iwahori subgroup with $K_\fp$.
\end{Definition}

In this terminology Example \ref{ExampleC2} says that for $\GG$ of type $C_2$ and split at $\fp$ the local Eichler algebra at $\fp$ is the full Hecke algebra. Seeing that for $V=k$ the trivial module we can identify $M(V,K)$ with the space of $k$-valued functions on $\GG(k) \backslash \GG(\hk) / K$, the elements in the Eichler subalgebra are characterized as those Hecke operators whose intrinsic combinatorics are already completely determined by the combinatorics of the chambers containing a given hyperspecial point.

\begin{Theorem}
 Let $\GG$ be simply connected and split at the finite prime $\fp$. The local Eichler algebra of $\GG$ at $\fp$ is a polynomial ring and the following table lists the number of indeterminates and the translations in the affine Weyl group whose double cosets generate the local Eichler algebra depending on the extended Dynkin diagram of $\GG$:
\begin{center}
\def\arraystretch{1.2}
\begin{tabular}{l|l|l}
Name &  Dynkin diagram & Generators \\
\hline
$\tilde{A}_n,n\geq 1$ & \An{0.5} &  $t\left(\omega_i^\vee + \omega_{n+1-i}^\vee\right),~1 \leq i \leq \left\lfloor \frac{n}{2}\right\rfloor$  \\
&& $\text{and } t(2\omega_{\frac {n+1} 2}^\vee)\text{ for odd n}$ \\
$\tilde{B}_n,n \geq 3$ & \Bn{0.5} &$t(\omega_{2i}^\vee),~1 \leq i \leq \left\lfloor \frac{n}{2} \right\rfloor,~\text{and } t(2\omega_1^\vee)$ \\
$\tilde{C}_n,n \geq 2$ & \Cn{0.5} &$t(\omega_i^\vee),1\leq i \leq n-1, \text{ and } t(2\omega_n^\vee)$ \\
$\tilde{D}_n,n \geq 4$, & \Dn{0.5} & $t(\omega_{2i}^\vee),~1 \leq i \leq \left\lfloor \frac{n}{2} \right\rfloor-1,~ t(2\omega_1^\vee),$  \\
$n$ even&&$t(2\omega_{n-1}^\vee) \text{ and } t(2\omega_{n}^\vee)$  \\
$\tilde{D}_n,n \geq 5$, &  \Dn{0.5} & $t(\omega_{2i}^\vee),~1 \leq i \leq \left\lfloor \frac{n}{2} \right\rfloor-1,~ t(2\omega_1^\vee),$  \\
$n$ odd& &$t(\omega_{n-1}^\vee+\omega_{n}^\vee)$  \\
$\tilde{E}_6$&\Esix{0.5} &$t(\omega_2^\vee), t(\omega_1^\vee+\omega_6^\vee)$  \\
$\tilde{E}_7$&\Eseven{0.5} &$t(\omega_1^\vee),t(\omega_5^\vee), t(2\omega_6^\vee)$ \\
$\tilde{E}_8$&\Eeight{0.5}\hspace{-11pt} &$t(\omega_1^\vee) , t(\omega_3^\vee)$\\
$\tilde{F}_4$&\Ffour{0.5} &$t(\omega_1^\vee), t(\omega_4^\vee)$ \\
$\tilde{G}_2$&\Gtwo{0.5} &$t(\omega_1^\vee)$ \\
\end{tabular}
\end{center}
Moreover, in the $C_n$-case the Eichler algebra coincides with the full Hecke algebra.
\end{Theorem}
\begin{proof}
 The proof of this theorem is a case-by-case check and since all cases work in a very similar manner we will only give the details for the $C_n$-case here.

 Let $\GG$ be of type $C_n$ (simply connected) and split at $\fp$. We can perform our computations locally and for ease of notation we set $F:=k_\fp$. We label the nodes of the extended Dynkin diagram as in the above table and set
\begin{equation}
 W_i:=\langle s_j ~|~0 \leq j \leq n, i \neq j \rangle < W_{af}.
\end{equation}
Choosing an appropriate Iwahori subgroup $I <K_\fp < \GG(F)$ we can assume $P_0:=K_\fp=I W_0 I$. Then the other maximal parahoric subgroups of $\GG(F)$ that share a common Iwahori subgroup with $K_\fp$ are represented by $P_i:=I W_i I,1\leq i \leq n$.

The Hecke algebra $H_{P_0}$ is generated by the characteristic functions of the double cosets $P_0 t(\omega_i^\vee) P_0,1 \leq i \leq n-1,$ and $P_0 t(2\omega_n^\vee)P_0$ whence it suffices to prove that these are also contained in the Eichler algebra. To that end we will show that the $W_0 $-double cosets that have a representative in $W_i$ are represented by $1,t(\omega_1^\vee),...,t(\omega_i^\vee)$ for $1 \leq i \leq n-1$ and by $1,t(\omega_1^\vee),...,t(\omega_{n-1}^\vee),t(2\omega_n^\vee)$ for $i=n$. The assertion then follows from Theorem \ref{VenkovElementAffine}.

We consider the classical realization of a root system of type $C_n$ in the Euclidean space $\MR^n$ with the standard inner product. The set of roots is
\begin{equation}
 \Phi:=\left\{\pm e_i\pm e_j~|~1 \leq i<j\leq n\right\} \cup \left\{\pm 2e_i~|~1\leq i \leq n \right\} 
\end{equation}
and the simple roots are
\begin{equation}
 \alpha_1=e_1-e_2,\alpha_2=e_2-e_3,...,\alpha_{n-1}=e_{n-1}-e_n,\alpha_n=2e_n.
\end{equation}
The coweights can be realized in the same vector space by the usual construction
\begin{equation}
 \alpha^\vee=\frac{2}{\langle \alpha,\alpha\rangle} \alpha.
\end{equation}
Under this identification we have $\omega_i^\vee=\sum_{j=1}^i e_j$ for $1 \leq i \leq n-1$ and $\omega_n^\vee=\frac{1}{2}\sum_{j=1}^n e_j$, in particular $\omega_1^\vee=\alpha_0^\vee$. We have $s_0=s_{\alpha_0}\omega_1^\vee$. If we denote the product $s_is_{i-1}...s_0$ by $s_{[i,0]}$ we compute
\begin{equation}
 W_0 s_0s_{[1,0]}...s_{[i-1,0]}= W_0 t(\omega_i^\vee),1 \leq i \leq n-1, \text{ and } W_0 s_0s_{[1,0]}...s_{[n-1,0]}=W_0t(2\omega_n^\vee).
\end{equation}
Moreover $s_0s_{[1,0]}...s_{[i-1,0]} \in W_i$ whence $W_0 t(\omega_i) W_0$ indeed has a representative in $W_j$ for $j\geq i$ (and $W_0 t(2\omega_n^\vee)W_0$ has a representative in $W_n$). On the other hand we have 
\begin{equation}
 |W_i\ds (W_0 \cap W_i)|=|\langle s_0,...,s_{i-1} \rangle \ds \langle s_1,...,s_{i-1}|=| C_2 \wr S_i \ds S_i | = i+1
\end{equation}
by using the standard isomorphism between a Coxeter group of type $C_i$ and the wreath product $C_2 \wr S_i$ which identifies the standard parabolic subgroup of type $A_{i-1}$ with the subgroup $S_i$. Hence 
\begin{equation}
 1,s_0,s_0s_1s_0,s_0s_{[1,0]}...s_{[i-1,0]}
\end{equation}
is indeed a full system of representatives of $W_i \ds (W_0 \cap W_i)$ which finishes the proof. 
\end{proof}

\section{Computational results}\label{resultssection}
In this section we present some of the computational results we achieved using the method outlined in the previous chapters. The programs used for these computations are available from the author's homepage.

\subsection{Reliability}
First we briefly want to discuss the reliability of our implementation.

The theory of modular forms offers a variety of plausibility checks for our results. When we enumerate a set of representatives of lattices in a given genus we have the mass formula which postulates that the inverses of the stabilizer orders should add up to a certain (precomputed) rational number. Moreover we often obtain several systems of representatives for the same genus from distinct computations which yields an additional check.\\ Furthermore two Hecke operators which are supported at distinct primes (or at the same primes $\fp$, where $K_\fp$ is hyperspecial) necessarily have to commute. This is a particularly strong check since the representing matrices may have - depending on the dimension of the space - several hundred entries, so the probability that two such matrices commute by chance is essentially zero. Moreover the Hecke operators we compute have a prescribed adjoint with respect to the Peterson scalar product (most of them ought to be self-adjoint). Our results passed all of these checks in several hundred sample computations we performed which should be seen as a strong indicator for the validity of our computations.
\subsection{Algebraic modular forms for symplectic groups}
Let $k$ be a totally real number field and $\MH$ a totally definite quaternion algebra over $k$. For $n \in \MN$ the $n \times n$-matrix ring over $\MH$ carries the natural involution $\dagger$ with $M^\dagger=\overline{M}^{tr}$ where $\overline{M}$ is the entrywise quaternionic conjugate of $M$. Now we set $\U_{n,\MH}$ the linear algebraic group over $k$ with $A$-rational points
\begin{equation}
 \U_{n,\MH}(A)=\{ g \in (A \otimes \MH)^{n \times n}|gg^\dagger=I_n \}
\end{equation}
for every commutative $k$-algebra $A$. Since $\MH$ is assumed to be definite the group $\U_{n,\MH}(k_\MR)$ is compact and moreover if $A \otimes \MH \cong A^{2 \times 2}$ (which happens for all but finitely many completions of $k$) we have $\U_{n,\MH}(A) \cong \Sp_{2n}(A)$, whence $\U_{n, \MH}$ is a compact form of $\Sp_{2n}$.

We now choose an $\CO_k$-maximal order $\CO_\MH$ in $\MH$ and consider open compact subgroups of $\U_{n,\MH}(\hk)$ arising from $\CO_\MH$-lattices in $\MH^n$.

Our tables (\ref{HeckeEigenvaluesSP6},\ref{HeckeEigenvaluesSP4Q5},\ref{HeckeEigenvaluesSP4Qtheta7}) give the decomposition of spaces of algebraic modular forms for symplectic groups arising in this way into Hecke eigenspaces. All these computations were performed with respect to trivial weight and the open compact subgroup defined by $\CO_\MH^n$ (corresponding to the so-called principal genus) which means we are only listing eigenvalues in the hyperspecial case where we can actually guarantee this decomposition. We describe the occurring quaternion algebras by their discriminant and the Hecke operators by the prime at which they are supported and the element in the affine Weyl group corresponding to the double coset. For nonrational eigenvalues we provide the minimal polynomial instead and for extensions of $\MQ$ we denote a prime ideal above $p \in \MZ$ by $\fp_p$.

{\small
\begin{table}[htb]
\begin{center}
\def\arraystretch{1.2}
\begin{tabular}{|l|l|l|l|}
\hline
 $\mathrm{disc}$ & $\mathrm{dim}$ & Operator & Eigenvalues \\
\hline
$3$ & $2$ & $h_2(s_0)$ & $126,9$\\
    &     & $h_2(s_0s_1s_0)$ & $2520,-54$\\
    &     & $h_2(s_0s_1s_2s_0s_1s_0)$ & $8640,216$\\
    &     & $h_5(s_0)$ & $19530,810$\\
    &     & $h_5(s_0s_1s_0)$ & $12694500,39780$\\
    &     & $h_5(s_0s_1s_2s_0s_1s_0)$ & $307125000,491400$\\
\hline
$5$ & $3$ & $h_2(s_0)$ & $126,33,-17$\\
    &     & $h_2(s_0s_1s_0)$ & $2520,226,76$\\
    &     & $h_2(s_0s_1s_2s_0s_1s_0)$ & $8640,456,-44$\\
    &     & $h_3(s_0)$ & $1092,100,0$\\
    &     & $h_3(s_0s_1s_0)$ & $98280,1064,364$\\
    &     & $h_3(s_0s_1s_2s_0s_1s_0)$ & $816480,7008,-1792$\\
\hline
$7$ & $5$ & $h_2(s_0)$ & $126,-3,-14,x^2 - 81x + 1512$\\
    &     & $h_2(s_0s_1s_0)$ & $2520,-18,70,x^2 - 708x + 92484$\\
    &     & $h_2(s_0s_1s_2s_0s_1s_0)$ & $8640,0,-110,x^2 - 1548x + 432864$\\
    &     & $h_3(s_0)$ & $1092,-4,-48,x^2 -208x + 2608$\\
    &     & $h_3(s_0s_1s_0)$ & $98280,-276,780,$\\&&&$x^2 - 3444x - 7969824$\\
    &     & $h_3(s_0s_1s_2s_0s_1s_0)$ & $816480,720,-1920,$\\&&&$x^2 - 26928x + 131341824$\\
    &     & $h_5(s_0)$ & $19530,-138,610,x^2 - 1440x + 467100$\\
    &     & $h_5(s_0s_1s_0)$ & $12694500,-72,37240$\\&&&$x^2 - 54780x - 419479200$\\
    &     & $h_5(s_0s_1s_2s_0s_1s_0)$ & $307125000,2448,203440,$\\&&&$x^2 - 941400x + 201197520000$\\
\hline
\end{tabular}
\caption{Hecke eigenvalues for $\Sp_6$ over $\MQ$.}\label{HeckeEigenvaluesSP6}
\end{center}
\end{table}
}

{\small
\begin{table}[htb]
\begin{center}
\def\arraystretch{1.2}
\begin{tabular}{|l|l|l|l|}
\hline
 $\mathrm{disc}$ & $\mathrm{dim}$ & Operator & Eigenvalues \\
\hline
$\fp_2\cdot \fp_5$& $4$ & $h_{\fp_3}(s_0)$ & $7380,580,180,-420$\\
    &     & $h_{\fp_3}(s_0s_1s_0)$ & $597780,12980,7380,4980$\\
    &     & $h_{\fp_{11}}(s_0)$ & $16104,-216,264,504$\\
    &     & $h_{\fp_{11}}(s_0s_1s_0)$ & $1948584,11944,16104,19384$\\
\hline
$\fp_2\cdot \fp_3$& $12$ & $h_{\fp_5}(s_0)$ & $0^{(2)},780,60,48,24^{(4)},-36,$ \\&&&$x^2- 84x - 11520$\\
    &     & $h_{\fp_5}(s_0s_1s_0)$ & $40^{(2)},19500,780,712,-200^{(4)},-460,$ \\&&&$x^2- 2100x + 770400$\\
    &     & $h_{\fp_{11}}(s_0)$ & $144^{(2)},16104,264,960,(x^2 - 528x + 58176)^{(2)},$ \\&&&$-216,x^2 + 264x - 35712$\\
    &     & $h_{\fp_{11}}(s_0s_1s_0)$ & $-2036^{(2)},1948584,16104,27820,$ \\&&&$ (x^2 - 5048x - 2027504)^{(2)},11944,$ \\&&&$x^2 - 25620x
 +161851104$\\
 \hline
\end{tabular}
\caption{Hecke eigenvalues for $\Sp_4$ over $\MQ\left(\sqrt{5}\right)$.}\label{HeckeEigenvaluesSP4Q5}
\end{center}
\end{table}
}
{\small
\begin{table}[htb]
\begin{center}
\def\arraystretch{1.2}
\begin{tabular}{|l|l|l|l|}
\hline
 $\mathrm{disc}$ & $\mathrm{dim}$ & Operator & Eigenvalues \\
\hline
$\fp_2$& $4$ & $h_{\fp_7}(s_0)$ & $2800,272,-112,-176$\\
    &     & $h_{\fp_7}(s_0s_1s_0)$ & $137200,4480,1792,1792$\\
    &    & $h_{\fp_{13}}(s_0)$ & $30940,-28,1092,-812$ \\
    &    & $h_{\fp_{13}}(s_0s_1s_0)$ &$5228860,26236,43836,21532$\\
    &    & $h_{\fp_3}(s_0)$ & $551880,-5544,2968,8568$ \\
    &    & $h_{\fp_3}(s_0s_1s_0)$ & $402320520,417816,595352,812952$ \\
\hline
\end{tabular}
\caption{Hecke eigenvalues for $\Sp_4$ over $\MQ\left(\zeta_7+\zeta_7^{-1}\right)$.}\label{HeckeEigenvaluesSP4Qtheta7}
\end{center}
\end{table}
}

\subsection{Runtime comparison}
Here we give a short comparison of the runtime for computing the action of the full (local) Hecke algebra by using Eichler elements compared to the standard approach following \cite{LanskyDecomposition} or \cite{LanskyPollack}. For this comparison we computed the full action of the local Hecke algebra at various primes acting on the space of algebraic modular forms of trivial weight and level defined by the principal genus of a two-dimensional space over various definite quaternion algebras (in the table represented by their discriminant). The computations were performed on an Intel core i7 processor running at 2.93 GHz. Note that while the Eichler method already comes out ahead in pure numbers it actually computes representatives for additional genera and additional Hecke operators acting on the corresponding spaces. 
{\small
\begin{table}[htb]
\begin{center}
\def\arraystretch{1.1}
\begin{tabular}{|l|l|l|l|}
\hline
 Discriminant & Prime & Runtime Standard & Runtime Eichler\\
\hline
$5$& $2$ & $39.32s$ & $9.83s$\\
   & $3$ & $295.93s$ & $21.93s$\\
\hline
$7$ & $2$ & $44.32s$ & $7.27s$\\
    & $3$ & $213.15s$ & $15.29s$\\
    & $5$ &$3295.69s$& $53.02s$\\
\hline
$11$ & $2$ & $139.75s$ &$34.60s$\\
    & $3$ & $713.32s$ & $69.30s$  \\
    & $5$ &$11344.21s$& $187.27s$\\
\hline
$13$ & $2$ & $81.96s$ & $28.55s$\\
    & $3$ & $634.54s$& $57.80s$\\
    & $5$ &$9017.63s$& $165.29s$\\
 \hline
\end{tabular}
\caption{Runtime comparison of the standard and Eichler method.}\label{RuntimeComparisonTable}
\end{center}
\end{table}
}
\clearpage
\section*{Acknowledgements}\label{ackref}
The results of this article are part of the author's PhD thesis which was supervised by Gabriele Nebe to whom the author wishes to express his deepest gratitude.

\bibliographystyle{abbrv}
\bibliography{Simul}

\begin{thebibliography}{10}

\bibitem{NebeBachoc}
C.~Bachoc and G.~Nebe.
\newblock Classification of two genera of {$32$}-dimensional lattices of rank
  {$8$} over the {H}urwitz order.
\newblock {\em Experiment. Math.}, 6(2):151--162, 1997.

\bibitem{CarterSimpleGroupsLieType}
R.~W. Carter.
\newblock {\em {Simple groups of Lie type}}, volume~22.
\newblock John Wiley \& Sons, 1989.

\bibitem{ASAPexceptional}
W.~K. Chan.
\newblock On almost strong approximation for some exceptional groups.
\newblock {\em Journal of Algebra}, 277(1):27--35, 2004.

\bibitem{ASAPclassical}
W.~K. Chan and J.~S. Hsia.
\newblock On almost strong approximation for algebraic groups.
\newblock {\em Journal of Algebra}, 254(2):441--461, 2002.

\bibitem{CohenNebePlesken}
A.~M. Cohen, G.~Nebe, and W.~Plesken.
\newblock Maximal integral forms of the algebraic group {$G_2$} defined by
  finite subgroups.
\newblock {\em Journal of Number Theory}, 72(2):282--308, 1998.

\bibitem{CunninghamDembeleGenus2}
C.~Cunningham and L.~Demb{\'e}l{\'e}.
\newblock Computing genus-2 {H}ilbert-{S}iegel modular forms over {$\mathds
  Q(\sqrt 5)$} via the {J}acquet-{L}anglands correspondence.
\newblock {\em Experiment. Math.}, 18(3):337--345, 2009.

\bibitem{EichlerAehnlichkeitsklassen}
M.~Eichler.
\newblock Die \"{A}hnlichkeitsklassen indefiniter {G}itter.
\newblock {\em Math. Z.}, 55:216--252, 1952.

\bibitem{GanYuGroupSchemes}
W.~T. Gan and J.-K. Yu.
\newblock Group schemes and local densities.
\newblock {\em Duke mathematical journal}, 105(3):497--524, 2000.

\bibitem{GreenbergVoightLatticeMethods}
M.~Greenberg and J.~Voight.
\newblock Lattice methods for algebraic modular forms on classical groups.
\newblock In {\em Computations with Modular Forms}, pages 147--179. Springer,
  2014.

\bibitem{GrossAlgebraicModularForms}
B.~H. Gross.
\newblock Algebraic modular forms.
\newblock {\em Israel J. Math.}, 113:61--93, 1999.

\bibitem{ASAPquadratic}
J.~S. Hsia and M.~J{\"o}chner.
\newblock Almost strong approximations for definite quadratic spaces.
\newblock {\em Inventiones mathematicae}, 129(3):471--487, 1997.

\bibitem{IwahoriMatsumotoBruhatDecomposition}
N.~Iwahori and H.~Matsumoto.
\newblock {On some Bruhat decomposition and the structure of the Hecke rings of
  p-adic Chevalley groups}.
\newblock {\em Publications Math{\'e}matiques de l'IH{\'E}S}, 25(1):5--48,
  1965.

\bibitem{KneserKlassenzahlenDefinit}
M.~Kneser.
\newblock Klassenzahlen definiter quadratischer {F}ormen.
\newblock {\em Arch. Math.}, 8:241--250, 1957.

\bibitem{LanskyDecomposition}
J.~Lansky.
\newblock Decomposition of double cosets in p-adic groups.
\newblock {\em Pacific Journal of Mathematics}, 197(1):97--117, 2001.

\bibitem{LanskyPollack}
J.~Lansky and D.~Pollack.
\newblock Hecke algebras and automorphic forms.
\newblock {\em Compositio Mathematica}, 130(1):21--48, 2002.

\bibitem{LoefflerExplicitComputations}
D.~Loeffler.
\newblock Explicit calculations of automorphic forms for definite unitary
  groups.
\newblock {\em LMS J. Comput. Math.}, 11:326--342, 2008.

\bibitem{PleskenSouvignier}
W.~Plesken and B.~Souvignier.
\newblock Computing isometries of lattices.
\newblock {\em Journal of Symbolic Computation}, 24(3):327--334, 1997.

\bibitem{SiegelAnalytischeTheorie}
C.~L. Siegel.
\newblock {\"U}ber die analytische {T}heorie der quadratischen {F}ormen.
\newblock {\em Annals of Mathematics}, pages 527--606, 1935.

\bibitem{TitsReductiveGroups}
J.~Tits.
\newblock Reductive groups over local fields.
\newblock In {\em Automorphic forms, representations and L-functions (Proc.
  Sympos. Pure Math., Oregon State Univ., Corvallis, Ore., 1977), Part},
  volume~1, pages 29--69, 1979.

\end{thebibliography}

\end{document}